\documentclass[]{interact}

\usepackage{epstopdf}
\usepackage[caption=false]{subfig}

\usepackage[numbers,sort&compress]{natbib}
\bibpunct[, ]{[}{]}{,}{n}{,}{,}

\theoremstyle{plain}
\newtheorem{theorem}{Theorem}[section]
\newtheorem{lemma}[theorem]{Lemma}

\newtheorem{proposition}[theorem]{Proposition}

\theoremstyle{definition}

\theoremstyle{remark}

\begin{document}

\title{Generating functions for vector partition functions and a basic recurrence relation}

\author{
\name{Alexander~P. Lyapin\textsuperscript{a}\thanks{CONTACT Alexander Lyapin. Email: aplyapin@sfu-kras.ru} and Sreelatha Chandragiri\textsuperscript{b}}
\affil{\textsuperscript{a}School of Mathematics \& Computer Science, Siberian  Federal  University, Svobodny, 79, Krasnoyarsk 660041, Russia; \textsuperscript{b}School of Mathematics \& Computer Science, Siberian  Federal  University, Svobodny, 79, Krasnoyarsk 660041, Russia}
}

\maketitle

\begin{abstract}
We define a generalized vector partition function and derive an identity for generating series of such functions associated with solutions of basic recurrence relation of combinatorial analysis. As a consequence we obtain the generating function of the number of generalized lattice paths and a new version of Chaundy-Bullard identity for the vector partition function.
\end{abstract}

\begin{keywords}
Generating function; basic recurrence relation; vector partition function; Chaundy-Bullard identity
\end{keywords}

\section{Statement of the main results}
A basic identity in the theory of summation of functions is the following, practically tautological,
relation:
\begin{equation}\label{eq1}
  \varphi(x) - \varphi(0) = \sum_{k=1}^x \left(\varphi(k) - \varphi(k-1)\right), x \in \mathbb N,
\end{equation}

If for a given function $h(x)$ it is possible to find a function $\varphi(x)$, such that $\varphi(x) - \varphi(x-1) = h(x)$, then \eqref{eq1}  becomes
\begin{equation}\label{eq2}
  \left.  \varphi(k) \frac{}{} \right|_0^x = \sum_{k=1}^x h(k),
\end{equation}
which is the discrete analogue of the Newton-Leibniz formula, and the function $\varphi(x)$  is a discrete primitive for the function $h(x)$.
For the generating function $\Phi(\xi) = \sum\limits_{x=0}^{\infty} \varphi(x) \xi^x, \xi\in\mathbb C$, identity \eqref{eq1} is equivalent to
\begin{equation}\label{eq3}
\Phi(\xi) - \frac{\Phi(0)}{1-\xi} = \frac 1{1-\xi}\sum_{x=1}^{\infty} \left( \varphi(x) - \varphi(x-1) \right) \xi^x
\end{equation}
or
\begin{equation}\label{eq4}
(1-\xi)\Phi(\xi) - \Phi(0)= \sum_{x=1}^{\infty} \left( \varphi(x) - \varphi(x-1) \right) \xi^x.
\end{equation}
In this paper we define a generalized vector partition function associated with a set of lattice vectors and with a complex-valued function of integer arguments. We give a multidimensional analogue of identity \eqref{eq4} in Theorem \ref{th1} and use it to investigate some properties of generalized lattice paths (Propositions \ref{prep1} and \ref{prep2}) and prove a version of the Chaundy-Bullard identity for the generalized vector partition function (Proposition \ref{prep3}, see \cite{ChaundyBullard1960}).

Let $\mathbb Z$ be the set of integers,  $\mathbb Z^n = \mathbb Z \times \cdots \times \mathbb Z$ and $\Delta = \{\alpha^1, \alpha^2, \ldots, \alpha^N\} \subset \mathbb Z^n$ be a finite set of column vectors. Let $\mathbb R^n_\geqslant$ be a subset of $\mathbb R^n$ with non-negative coordinates. We let K denote the cone $K$ spanned by the vectors in $\Delta$:
$$
K = \{\lambda \in \mathbb R^n : \lambda =   x_1 \alpha^1 + \cdots + x_N \alpha^N, x \in \mathbb R^N_\geqslant\}.
$$
We assume that cone $K$ is \emph{pointed}, which means it does not contain any line or equivalently lies in an open half-space of $\mathbb R^n$.  We let  $A = [\alpha^1, \ldots, \alpha^N]$ denote $(n \times N)$-matrix composed of the column vectors in $\Delta$.

The vector partition function  $P_A(\lambda)$ of $\lambda \in \mathbb  Z^n$ is (see, for example, \cite{Stanley1990}) the number of non-negative integer solutions to a linear Diophantine equation $\alpha^1 x_1 + \ldots + \alpha^N x_N = \lambda$:
\begin{equation}\label{eq5}
P_A(\lambda) = \sum_{\substack{ x: Ax=\lambda\\ x \in \mathbb Z^N_\geqslant}}1, \;\;\lambda \in \mathbb Z^n.
\end{equation}

Geometrically the function $P_A(\lambda)$ equals the number of representations of the vector $\lambda$ as a linear combination of the vectors in $\Delta$ with non-negative integer coefficients.

In \cite{BrionVergne1997} properties of the function
\begin{equation}\label{eq6}
P_A(y;\lambda) = \sum_{\substack{ x: Ax=\lambda\\ x \in \mathbb Z^N_\geqslant}} e^{-\langle x,y \rangle}, y \in \mathbb C^N,
\end{equation}
called \emph{the vector partition function associated with the set of vectors $\Delta$}, are investigated. In particular, they derive the residue formulas for its generating function and an analog of the Euler-Maclaurin formula, in which the vector partition functions are represented as the action of the Todd operator on the volume function of a polyhedron.

Furthermore, a sum of $e^{\langle x,y \rangle}$ in integer cones was investigated in \cite{PukhlikovKhovanskii1993} in connection to generalization of the Riemann-Roch theorem. A structure theorem for vector partition function was presented and polyhedral tools for the efficient computation of such functions was provided in \cite{Sturmfels1995}.

For an arbitrary function of integer arguments $\varphi: \mathbb Z^N_\geqslant \to \mathbb C$ we define a function
$$
P_A(\lambda;\varphi) = \sum_{ \substack {x : Ax = \lambda\\ x \in \mathbb Z^N_\geqslant}} \varphi(x), \lambda \in \mathbb Z^n
$$
which we call \emph{the vector partition function associated with $\varphi(x)$}.

For $\varphi(x) \equiv 1$, then the function $P_A(\lambda;\varphi)$  coincides with the classical function of the vector partition \eqref{eq5}. For $\varphi(x) = e^{-\langle x, y\rangle}$  we obtain a vector partition function of the form \eqref{eq6}. If we take $N=2, A=
\left(
\begin{array}{cc}
     1 & 1 \\
\end{array}
\right)
$
and $\varphi(x_1,x_2) = h(x_1)$, then $P_A(\lambda; \varphi) = \sum\limits_{\substack{x_1+x_2 = \lambda\\ x_1, x_2 \geqslant 0}} h(x_1) = \sum\limits_{x_1 = 0}^{\lambda} h(x_1)$.
Thus, the problem of finding the vector partition function $P_A(\lambda; \varphi)$  is a generalization of the classical summation's problem of functions of a discrete argument.

For further discussion and formulation of the main result we introduce some notations. Let $V = \{J\}$ be a set of all ordered sets $J = (j_1,j_2, \ldots, j_k)$, $1 \leqslant j_1 < \ldots < j_k \leqslant N, k = 0, 1, 2, \ldots, N$ and $\#J = k$ be a number of elements in the set $J$. Let $\pi_j$ be the projection operator along the $j$-th coordinate axis in $\mathbb R^n$, i.e. $\pi_j x = (x_1, \ldots, x_{j-1}, 0, x_{j+1}, \ldots, x_N)$, and define its adjoint action on the function $\varphi(x): \mathbb Z^N \to \mathbb C$ by:
$
\pi_j \varphi(x) = \varphi(\pi_j x), j = 1, \ldots, N.
$
Let $\mathbb C [[\xi]]$ be the ring of formal power series in the variable $\xi = (\xi_1, \ldots, \xi_N)$  and define the operator $\pi_j: \mathbb C [[\xi]] \to \mathbb C [[\xi]]$ for $j = 1, \ldots, N$ on the generating series $\Phi(\xi) = \sum\limits_{x \in \mathbb Z^N} \varphi(x) \xi^x$
as follows
$$
\pi_j \Phi(\xi) = \sum_{x \in \mathbb Z^N} \varphi(\pi_j x) \xi^{\pi_j x} = \Phi(\xi_1, \ldots, \xi_{j-1}, 0, \xi_{j+1}, \ldots, \xi_N).
$$
Furthermore, for $J \in V$ let $\pi_J = \pi_{j_1} \circ \cdots \circ \pi_{j_k}$ be a composition of operators
$\pi_{j_1}, \ldots, \pi_{j_k}$ and $\pi_\emptyset = 1$ is the the identity operator.

For complex-valued functions $\varphi(x)$ of integer arguments $x = (x_1, \ldots, x_N)$ we define a shift operator $\delta_j$ for $j = 1, \ldots, N$ as follows $$\delta_j \varphi(x) = \varphi(x_1, \ldots, x_{j-1}, x_j+1, x_{j+1}, \ldots, x_N)$$ and for $\mu = (\mu_1, \ldots, \mu_N) \in \mathbb Z^N$ we define $\delta^\mu = \delta_1^{\mu_1} \cdots \delta_N^{\mu_N}$.

For $c = (c_1, \ldots, c_N)\in \mathbb C^N$  we denote the operator $Q(\delta) = \delta_1 \cdot \ldots \cdot \delta_N - c_1 \delta_2 \cdot \ldots \cdot \delta_N - \cdots - c_N\delta_1 \cdot \ldots \cdot \delta_{N-1}$ and $z^A = (z^{\alpha^1}, \ldots, z^{\alpha^N})$.

Now we formulate a multidimensional analogue of identity (4).

\begin{theorem}\label{th1}
Let $P_A(\lambda;\varphi)$ be the vector partition function, associated with a function $\varphi: \mathbb Z^N_\geqslant \to \mathbb C$, and $\Phi(\xi)$ be the generating series for $\varphi(x)$. Then
\begin{equation}\label{th1formula}
\sum_{J\in V} (-1)^{\#J} \pi_J \left. \left[(1- \langle c, \xi\rangle) \Phi(\xi) \frac{}{}\right] \right|_{\xi = z^A} = \sum_{\lambda \in K \cap \, \mathbb Z^N} P_A(\lambda;Q(\delta) \varphi) z^\lambda.
\end{equation}
\end{theorem}

We note that for $N = 1$ and $c = 1$, formula \eqref{th1formula} implies identity \eqref{eq4}.

The problem of finding \emph{the number of generalized lattice paths} is formulated as follows:
find the number of paths on an integer lattice from the origin to the point $\lambda \in \mathbb Z^n$ using only steps in $\Delta$.

A similar problem for $n=2$ connected with the study of multidimensional difference equations and its application for generalized Dyck paths (see \cite{LabelleYeh1990}) was considered in \cite{BousquetMelouPetkovsek2000}. We note that the Cauchy problem for multidimensional difference equations was considered also in \cite{Leinartas2007}, \cite{LeinartasRogozina2015} and some properties of generating function of its solution were investigated in \cite{LeinartasLyapin2009}.

A simple case of the problem above arises by choosing the set of steps which form an orthonormal basis in $\mathbb R^N$: $\alpha^j = e^j,  j = 1, \ldots, N$, and the number $\varphi(x)$ of paths from the origin to the point $x \in \mathbb Z^N$ satisfies the basic recurrence relation of combinatorial analysis
\begin{equation}\label{basicrecurrencerelation}
(1-\langle I, \delta^{-I} \rangle) \varphi(x) \equiv \varphi(x) - \varphi(x-e^1) - \cdots - \varphi(x-e^N) = 0
\end{equation}
where $x \in I + \mathbb Z^N_\geqslant$. For any solution of (\ref{basicrecurrencerelation}) we have:

\begin{proposition}\label{prep1}
If the function $\varphi(x)$ satisfies equation \eqref{basicrecurrencerelation}, then the associated vector partition function $P_A(\lambda;\varphi)$ satisfies the difference equation
$$
(1- \langle I, \delta_\lambda^{-A}\rangle) P_A(\lambda;\varphi) \equiv P_A(\lambda; \varphi) - \sum_{j=1}^{N} P_A(\lambda-\alpha^j;\varphi)= 0.
$$
\end{proposition}

The following proposition relates the number of lattice paths $\varphi(x)$ and the number of generalized lattice paths $P_A(\lambda;\varphi)$:

\begin{proposition}\label{prep2}
If $\varphi(x)$ is the number of lattice paths, then the associated the vector partition function $P_A(\lambda;\varphi)$ coincides with the number of generalized lattice paths with steps in $\Delta$; in this case its generating function $F(z) = \sum\limits_{\lambda\in K \cap \mathbb Z^n} P_A(\lambda;\varphi)z^\lambda$ has the form
$$
  F(z) = \dfrac 1 {1 - z^{\alpha^1} - z^{\alpha^2} - \cdots - z^{\alpha^N}}.
$$
\end{proposition}

In 1960 T. Chaundy and J. Bullard in \cite{ChaundyBullard1960} considered the identity which is valid for $c_1, c_2 \in \mathbb C$ such that $c_1+c_2=1$ and nonnegative integers $\mu_1$  and $\mu_2$
$$
c_2^{\mu_2+1} \sum_{\nu_1=0}^{\mu_1} \binom{\mu_1+\mu_2-\nu_1}{\mu_1-\nu_1} c_1^{\mu_1-\nu_1} +
c_1^{\mu_1+1} \sum_{\nu_2=0}^{\mu_2} \binom{\mu_1+\mu_2-\nu_2}{\mu_2-\nu_2} c_2^{\mu_2-\nu_2} \equiv 1.
$$

This identity was subsequently found in approximation theory, nonrecursive digital filters \cite{Herrmann1971}, in the theory of wavelets \cite{Daubeacheis1992}, in the theory of Gauss hypergeometric functions. A detailed review, including various proofs of a one-dimensional case and a multidimensional analogues of this identity was given in \cite{KoornwinderSchlosser2008}, \cite{KoornwinderSchlosser2013} and \cite{Egorychev2011}. In \cite{KrivokoleskoLeinartas2012} similar identities were derived by using methods of generating functions and properties of Hadamars composition of multiple power series.

We give an analogue of the Chaundy-Bullard identity for vector partition function. For $j=1, \ldots, N$ we denote $\Delta_j = \Delta \setminus \{\alpha^j\}$ and $A_j = [\alpha^1, \ldots, \alpha^{j-1}, \alpha^{j+1}, \ldots, \alpha^N]$, then $K_j = \{\nu\in\mathbb Z^n : \nu = y_1 \alpha^1 + \ldots [j] \ldots + y_N \alpha^N, y\in\mathbb Z^N_\geqslant$\}; $c^x = c_1^{x_1} \cdots c_N^{x_N}$. Note that each cone $K_j \subset K$ is also pointed.

\begin{proposition}\label{prep3}
If $c_1 + c_2 + \cdots + c_N =1$ and $\varphi_j(x) = \frac{|x|!}{x!}\, c^{\,x+e^j}$, then for any $\mu = (\mu_1, \ldots, \mu_N) \in \mathbb Z^N$  the identity
\begin{equation}\label{prep3form_new}
 \sum_{j=1}^{N} \sum_{\substack{\nu \in K_{j}}} P_{A_j}(\nu) P_A(\mu - \nu;\varphi_j) = P_A (\mu)
\end{equation}
takes place.
\end{proposition}
The sum on the left side of identity (\ref{prep3form_new}) is finite since all the cones $K_{j}, j=1, \ldots, N$ are pointed.

Note that for $\alpha^j = e^j, j=1, \ldots, N$ we obtain from \eqref{prep3form_new} a multidimensional Chaundy-Bullard identity:
\begin{equation*}
      \sum_{j=1}^{N} \sum_{\substack{0\leqslant \nu \leqslant \mu \\ \nu_j = 0}} \frac{(|\mu| - |\nu|)!} {(\mu - \nu)!} c^{\: \mu - \nu + e^j} \equiv 1,
\end{equation*}
where the double inequality $0\leqslant \nu \leqslant \mu$ means that $0\leqslant \nu_j \leqslant \mu_j$ for all $j = 1, \ldots, N$.

\section{Proofs}

In this section we prove the main result (see Theorem \ref{th1}) and its corollaries. First, using Lemmas \ref{lemma1} and \ref{lemma2}, we prove the identity \eqref{th1formula} for a set of standard basis vectors, and then, using a monomial substitution, we prove Theorem
\ref{th1}. The proofs of Propositions \ref{prep1} and \ref{prep2} follow directly from the main result. Then, using Lemmas \ref{lemma3} and \ref{lemma4}, we prove Proposition \ref{prep3}.

We prove identity \eqref{th1formula} for the case when the set of vectors $\Delta = \{\alpha^1, \ldots, \alpha^N \}$ consists of unit vectors $\alpha^j = e^j$, where the vector $e^j = (0, \ldots, 0, 1, 0, \ldots, 0)$  contains a unit on the $j$-th place for $j = 1, \ldots, N$. Then the vector partition function $P_A(\lambda;\varphi) = \varphi(x)$, and the generating series $\Phi(\xi) = \sum\limits_{x \in \mathbb Z^N_\geqslant} \varphi(x) \xi^x$ in identity  \eqref{th1formula} takes the form
\begin{equation}\label{th2formula}
\sum_{J \in V} (-1)^{\#J} \pi_J \left[ (1- \langle c,\xi \rangle) \Phi(\xi)\right] = \sum_{x \in I + \mathbb Z^N_\geqslant} (1-\langle c, \delta^{-I} \rangle) \varphi(x) \xi^x,
\end{equation}
where $\langle c, \xi \rangle = c_1 \xi_1 + \cdots + c_N \xi_N$, and $I = (1, \ldots, 1)$.

To prove \eqref{th2formula} we use following properties of the operator $\mathbf \Pi =\sum\limits_{J\in V}  (-1)^{\# J} \pi_J$:

\begin{lemma}\label{lemma1}
For an arbitrary series $\Phi(\xi) = \sum\limits_{x\in \mathbb Z^N_\geqslant} \varphi(x) \xi^x$ the operator $\mathbf \Pi$ acts on $\Phi(\xi)$ as follows
$$
\mathbf \Pi: \sum\limits_{x\in \mathbb Z^N_\geqslant} \varphi(x) \xi^x \mapsto \sum\limits_{x\in I + \mathbb Z^N_\geqslant} \varphi(x) \xi^x.
$$
\end{lemma}

\begin{proof}
We represent the operator $\mathbf \Pi$ as a composition $\mathbf \Pi = (1-\pi_1)(1-\pi_2) \cdots (1-\pi_N)$, where $1 = \pi_\emptyset$ is an identity operator,  and use the commutativity of its factors to apply it to the series $\Phi(\xi)$:
\begin{multline*}
(1-\pi_1)(1-\pi_2) \cdots (1-\pi_N)  \Phi(\xi) = \\
= (1-\pi_1)(1-\pi_2) \cdots (1-\pi_{N-1}) \left[ \Phi(\xi) - \Phi(\pi_N \xi) \right]=\\
= (1-\pi_1)(1-\pi_2) \cdots (1-\pi_{N-1})  \sum\limits_{ x\in e^N + \mathbb Z^N_\geqslant} \varphi(x) \xi^x = \\
= (1-\pi_1)(1-\pi_2) \cdots (1-\pi_{N-2})  \sum\limits_{ x\in e^N + e^{N-1} + \mathbb Z^N_\geqslant} \varphi(x) \xi^x = \\
= \ldots = \sum\limits_{ x\in I + \mathbb Z^N_\geqslant} \varphi(x) \xi^x,
\end{multline*}
where $I = e^1 + e^2 + \ldots + e^N = (1, 1, \ldots, 1)$.
\end{proof}

\begin{lemma}\label{lemma2}
If $\mathbf \Pi_j = (1-\pi_1) \cdots (1-\pi_{j-1}) (1-\pi_{j+1}) \cdots (1-\pi_N)$, then for any $j = 1, \ldots, N$ the following equality holds
$$
    \mathbf \Pi \xi_j \Phi(\xi) = \mathbf \Pi_j \xi_j \Phi(\xi) = \sum\limits_{x \in I + \mathbb Z^N_\geqslant} \varphi(x - e^j) \xi^x.
$$
\end{lemma}
\begin{proof}
Similarly as in the proof of Lemma \ref{lemma1} we represent the operator  $\mathbf \Pi$ as a composition and apply it to $\xi_j \Phi(\xi) \in \mathbb C[[\xi]]$:
    \begin{multline*}
      (1-\pi_1)(1-\pi_2) \cdots (1-\pi_N) \left[  \xi_j \Phi(\xi)\right] = \\
      = (1-\pi_1) \cdots (1-\pi_{j-1}) (1-\pi_{j+1}) \cdots (1-\pi_N) \left[ (1-\pi_j) \xi_j \Phi(\xi)\right] = \\
      = (1-\pi_1) \cdots (1-\pi_{j-1}) (1-\pi_{j+1}) \cdots (1-\pi_N) \left[ \xi_j \Phi(\xi) - \pi_j\xi_j \Phi(\xi)\right] = \\
      = (1-\pi_1) \cdots (1-\pi_{j-1}) (1-\pi_{j+1}) \cdots (1-\pi_N)  \left[ \xi_j \Phi(\xi)\right] = \\
      = \mathbf \Pi_j   \sum\limits_{x \in e^j + \mathbb Z^N_\geqslant} \varphi(x - e^j) \xi^x
      =  \sum\limits_{x \in I + \mathbb Z^N_\geqslant} \varphi(x - e^j) \xi^x.
    \end{multline*}
\end{proof}

By using Lemma \ref{lemma1} and \ref{lemma2}, we now prove \eqref{th2formula}.

\begin{proof}
We apply the operator  $\mathbf \Pi$ to the product $(1- \langle c,\xi \rangle) \Phi(\xi)$ and use Lemmas \ref{lemma1} and \ref{lemma2} to obtain:
  \begin{multline*}
    \mathbf \Pi \left[ (1- \langle c,\xi \rangle) \Phi(\xi) \right] = \mathbf \Pi\Phi(\xi) - \mathbf \Pi \left[\langle c,\xi \rangle \Phi(\xi)\right] = \mathbf \Pi\Phi(\xi) - \langle  c, \mathbf\Pi \xi \rangle \Phi(\xi) = \\
    = \mathbf \Pi\Phi(\xi) - c_1 \mathbf\Pi_1 \xi_1 \Phi(\xi) - \cdots - c_N \mathbf\Pi_N \xi_N \Phi(\xi) = \\
    = \sum\limits_{x \in I + \mathbb Z^N_\geqslant} \varphi(x) \xi^x  - c_1 \sum\limits_{x \in I + \mathbb Z^N_\geqslant} \varphi(x - e^1) \xi^x - \ldots - c_N \sum\limits_{x \in I + \mathbb Z^N_\geqslant} \varphi(x-e^N) \xi^x = \\
    = \sum\limits_{x \in I + \mathbb Z^N_\geqslant} \left[\varphi(x) - c_1 \varphi(x - e^1) - \ldots - c_N \varphi(x-e^N)\right] \xi^x = \\
    = \sum_{x \in I + \mathbb Z^N_\geqslant} \left[(1-\langle c, \delta^{-I} \rangle) \varphi(x)\right] \xi^x.
  \end{multline*}
\end{proof}

Now we proceed to the proof of the main theorem.
\begin{proof}
We substitute the monomial replacement of the variables $\xi = z^A$ in the formula \eqref{th2formula} to obtain
\begin{multline*}
\xi^x = (z_1^{\alpha_1^1} \cdots z_n^{\alpha_n^1})^{x_1} \ldots (z_1^{\alpha_1^N} \cdots z_n^{\alpha_n^N})^{x_N} = z_1^{x_1 \alpha_1^1 + \ldots + x_N \alpha_1^N} \cdots z_n^{x_1 \alpha_n^1 + \ldots + x_N \alpha_n^N} = z^\lambda,
\end{multline*}
where $\lambda = Ax$. We further observe that if $x \in I + \mathbb Z^N_\geqslant$, then $\lambda \in K\cap \mathbb Z^n$.
Therefore,
\begin{multline*}
 \sum_{J \in V} \left.(-1)^{\#J} \pi_J \left[ (1- \langle c,\xi \rangle) \Phi(\xi)\right] \frac{}{} \right|_{\xi = z^A} = \\
 = \sum\limits_{x \in I + \mathbb Z^N_\geqslant} \left.\left[\varphi(x) - c_1 \varphi(x - e^1) - \ldots - c_N \varphi(x-e^N)\right] \xi^x \frac{}{} \right|_{\xi = z^A} = \\
 = \sum_{\lambda \in K\cap\, \mathbb Z^n} \sum_{\substack{Ax = \lambda \\ x \in \mathbb Z^N_\geqslant}} \left[\varphi(x+I) - c_1 \varphi(x+I - e^1) - \ldots - c_N \varphi(x + I - e^N)\right] z^\lambda = \\
 = \sum_{\lambda \in K\cap\, \mathbb Z^n} P_A \left(\lambda; Q(\delta)\varphi \right) z^\lambda,
\end{multline*}
where $Q(\delta) = \delta_1 \cdot \ldots \cdot \delta_N - c_1 \delta_2 \cdot \ldots \cdot \delta_N - \cdots - c_N\delta_1 \cdot \ldots \cdot \delta_{N-1}$.
\end{proof}

Now we prove Proposition \ref{prep1}.

\begin{proof}
Summing the left side of the basic recurrence relation (\ref{basicrecurrencerelation}) over all integer nonnegative $x: Ax = \lambda$, we obtain
\begin{equation}\label{prep1.proof}
  \sum_{x: Ax = \lambda}\varphi(x) - \sum_{x: Ax = \lambda}\varphi(x-e^1) - \cdots -\sum_{x: Ax = \lambda}\varphi(x-e^N) = 0.
\end{equation}

Note that for any $j = 1, \ldots, N$ we have
\begin{gather*}
  \sum_{\substack{x: Ax = \lambda \\ x\geqslant 0}}\varphi(x-e^j) = \sum_{\substack{x: A(x - e^j + e^j) = \lambda \\ x\geqslant 0}} \varphi(x-e^j)
  = \sum_{\substack{x: A(x - e^j) = \lambda - \alpha^j \\ x\geqslant 0}} \varphi(x-e^j) = \\ = \sum_{\substack{x: Ax = \lambda - \alpha^j \\ x\geqslant -e^j}} \varphi(x) =  \sum_{\substack{x: Ax = \lambda - \alpha^j \\ x\geqslant 0}} \varphi(x) -  \sum_{\substack{x: Ax = \lambda - \alpha^j \\ x_j = -e^j}} \varphi(x) = \\ = P_A(\lambda-\alpha^j;\varphi),
\end{gather*}
since $\varphi(x) = 0$ for all $x \notin \mathbb Z^N_\geqslant$.
Then from \eqref{prep1.proof} we obtain
\begin{equation*}
  P_A(\lambda;\varphi) - P_A(\lambda-\alpha^1;\varphi) - \ldots - P_A(\lambda-\alpha^N;\varphi) = 0.
\end{equation*}
\end{proof}

Now we prove Preposition \ref{prep2}.

\begin{proof}
We prove that the generating function $\Phi(\xi)$ of the number of lattice paths is equal to $\Phi(\xi) =  (1-\xi_1-\xi_2-\ldots-\xi_N)^{-1}$. Indeed, the number $\varphi(x)$ of lattice paths satisfies the basic recurrence relation (\ref{basicrecurrencerelation}). Therefore, the right side of (\ref{th2formula}) vanishes. Now we use induction on the number $N$ of variables $\xi = (\xi_1,\xi_2, \ldots, \xi_N)$. For $N=1$ formula \eqref{th2formula} takes the form $(1-\xi)\Phi(\xi) -1 =0$, where $\Phi(\xi) = (1-\xi)^{-1}$ and for any number $m<N$ of variables $(\xi_{i_1}, \xi_{i_2}, \ldots, \xi_{i_m})$ the generating function $\Phi(\xi_{i_1}, \xi_{i_2}, \ldots, \xi_{i_m}) = (1 - \xi_{i_1} - \xi_{i_2} - \cdots - \xi_{i_m})^{-1}$.

We note that for any $J = \{j_1, j_2, \ldots, j_k\}, k \leqslant N, J \neq \emptyset$ the number of lattice paths in $\mathbb Z^{N-k} = \mathbb Z^N \cap \{x_{j_1} = \cdots = x_{j_k} = 0\}$  can be written as a function $\pi_J \varphi(x)$, and the induction hypothesis implies that its generating function $\Phi(\xi)$ satisfies the relation $\pi_J \Phi(\xi) = \pi_J (1-\xi_1-\cdots - \xi_N)^{-1}$ or $\pi_J \left[(1-\xi_1-\cdots - \xi_N) \Phi(\xi)\right] = 1$.

Next, we select in \eqref{th2formula} the term corresponding to $J =\emptyset$:
\begin{equation*}
  (1-\xi_1 - \ldots - \xi_N) \Phi(\xi) + \sum_{\substack{J \neq \emptyset \\ J \in V}} (-1)^{\# J} 1 = 0.
\end{equation*}
The equality  $\sum\limits_{\substack{J \neq \emptyset \\ J \in V}} (-1)^{\# J} 1 = - C^N_{N-1} + C^N_{N-2} + \cdots + (-1)^N C^N_0 = -1$ implies that the induction statement is also true for $m = N$. After making the substitution $\xi = z^A$ we obtain Proposition \ref{prep2}.
\end{proof}
The following two lemmas are required to prove the Proposition \ref{prep3}.

\begin{lemma}\label{lemma3}
If the function $\Phi(\xi)$ does not depend on the variable $\xi_j$, then ${\bf\Pi} \Phi(\xi) = 0$.
\end{lemma}

\begin{proof}
We consider the operator ${\bf\Pi}_j = (1-\pi_1) \circ \cdots {[j]} \cdots \circ (1-\pi_N), j = 1, \ldots, N$. Since the function $\Phi(\xi)$ does not depend on the variable $\xi_j$, we have $(1-\pi_j) \Phi(\xi) = \Phi(\xi) - \pi_j \Phi(\xi) = 0$, therefore ${\bf\Pi} \Phi(\xi) = {\bf\Pi}_j (1-\pi_j) \Phi(\xi) = 0.$
\end{proof}

\begin{lemma}\label{lemma4}
For any complex values $c_1, \ldots, c_N$ such that $c_1 + \cdots + c_N = 1$,  the rational fraction $(I - \xi)^{-1}$ can be represented as follows:
\begin{equation}\label{lemma4form}
     \frac {1}{I - \xi} = \sum_{j=1}^{N} \frac{c_j}{\pi_j (I - \xi)} \cdot \frac{1}{1-\langle c, \xi \rangle}.
\end{equation}
\end{lemma}

\begin{proof}
 Since $c_1 + \cdots + c_N = 1$ and $\varphi(x) \equiv 1$, the right side of \eqref{th2formula} vanishes, and its left side can be written as follows:
$$
(1-\langle c, \xi \rangle) \frac 1{I - \xi} + \sum_{J\neq \emptyset} (-1)^{\#J} \pi_J (1-\langle c, \xi \rangle)\Phi(\xi) = 0.
$$
Since $\sum\limits_{J\neq \emptyset} (-1)^{\#J} \pi_J  = {\bf\Pi} - 1$ and $ (1-\langle c, \xi \rangle) = \sum\limits_{j=1}^{N} c_j (1-\xi_j)$, we obtain
\begin{equation*}
\frac{1-\langle c, \xi \rangle}{I - \xi} + \sum_{j=1}^{N} {\bf\Pi} \frac{c_j}{\pi_j (I - \xi)} = \sum_{j=1}^{N} \frac{c_j}{\pi_j (I - \xi)}.
\end{equation*}
Furthermore, Lemma \ref{lemma3} implies
\begin{equation*}
 {\bf \Pi} \left( \pi_j \frac{c_j}{I-\xi}\right) = 0, j = 1, \ldots, N,
\end{equation*}
which completed the proof.
\end{proof}

Now we prove the Chaundy-Bullard identity for the vector partition function (Proposition \ref{prep3}).

\begin{proof}
We expand both sides of \eqref{lemma4form} in power series in $\xi^\mu$ to obtain
\begin{equation}\label{proofprep3f1}
   \sum_{\mu\in\mathbb Z^N_\geqslant} \xi^\mu = \sum_{j=1}^{N} \frac{c_j}{\pi_j (I - \xi)} \sum_{\mu\in\mathbb Z^N_\geqslant} \frac{c^\mu |\mu|!}{\mu!}\, \xi^\mu.
\end{equation}
We recall that
$
K = \{\mu\in\mathbb Z^n : \mu = x_1 \alpha^1 + \ldots + x_N \alpha^N, x\in\mathbb Z^N_\geqslant\},
K_{j} = \{\nu\in\mathbb Z^n : \nu = y_1 \alpha^1 + \ldots [j] \ldots + y_N \alpha^N, y\in\mathbb Z^N_\geqslant\},
$
and $ K_{j} \subset K$ for $j = 1, \ldots, N$.  In \eqref{proofprep3f1} we substitute $\xi = z^A$ and transform the left side as follows
\begin{equation*}
 \frac{1}{I - z^A} = \sum_{x\in \mathbb Z^n_\geqslant} z^{Ax}= \sum_{\mu \in K} \left(\sum_{\substack{x : Ax = \mu \\ x\in \mathbb Z^n_\geqslant}} 1 \right) z^\lambda = \sum_{\mu \in K} P_A(\mu) z^\mu.
\end{equation*}

We denote $\varphi_j(x) = \frac{|x|!}{x!}\,c^{x+e^j}$, then the right side of \eqref{proofprep3f1} takes the form
\begin{multline*}
  \sum_{j=1}^{N} c_j \pi_j (I - \xi)^{-1} (1- \langle c,\xi \rangle)^{-1} = \sum_{j=1}^{N} \left( \sum_{\substack{y \geqslant 0 \\ y_j =0}} \xi^y \cdot \sum_{x\in\mathbb Z^N_\geqslant} \varphi_j(x) \right) =\\
  = \sum_{j=1}^{N} \left( \sum_{\nu \in K_{j}} \left( \sum_{\substack{y : Ay = \nu\\y_j=0}} 1 \right) z^\nu \cdot \sum_{\lambda\in K} \left( \sum_{x : Ax = \lambda}\varphi_j(x) \right) z^\lambda \right) = \\
  =\sum_{j=1}^{N} \left(  \sum_{\nu \in K_{j}} P_{\Delta_j}(\nu) z^\nu  \cdot  \sum_{\lambda \in K} P_{\Delta}(\lambda; \varphi_j(x)) z^\lambda \right) =\\
  = \sum_{\mu \in K} \left( \sum_{j=1}^{N} \sum_{\substack{\nu+\lambda = \mu\\ \nu \in K_{j}\\ \lambda\in K}} P_{\Delta_j}(\nu) P_\Delta(\lambda;\varphi_j(x)) \right) z^\mu.
\end{multline*}

Equating the coefficients of $z^\mu$ yields
\begin{equation*}
  P_A (\mu) = \sum_{j=1}^{N} \sum_{\substack{\nu+\lambda = \mu\\ \nu \in K_{j}\\ \lambda\in K}} P_{A_j}(\nu) P_A(\lambda;\varphi_j(x)) = \sum_{j=1}^{N} \sum_{\substack{\nu \in K_{j}}} P_{A_j}(\nu) P_A(\mu - \nu;\varphi_j(x)),
\end{equation*}
which completes the proof.
\end{proof}

\section*{Acknowledgements}

\emph{The authors express their gratitude to E.K. Leinartas for guidance and permanent interest to the work. The authors thank W.M. Lawton for helpful comments which significantly improved this paper.}

\section*{Funding}

\emph{The second author was supported by the PhD SibFU grant for support of scientific research No: 14.}


\begin{thebibliography}{99}

\bibitem{BousquetMelouPetkovsek2000}
M. Bousquet-Melou, M. Petkovsek, \emph{Linear recurrences with constant coefficients: the multivariate case}, Discrete Mathematics, 225 (2000), pp. 51--75.

\bibitem{BrionVergne1997}
M. Brion, M. Vergne, \emph{Residue formulae, vector partition functions and lattice points in rational polytopes}, J. American Math.Soc., Vol. 10, no. 4 (1997), pp. 797--833.

\bibitem{ChaundyBullard1960}
T.W. Chaundy, J.E. Bullard, \emph{John Smith's problem}, Math. Gazette, Vol. 44, (1960), pp. 253--260.

\bibitem{Daubeacheis1992}
I. Daubeacheis, \emph{Ten Lectures on Wavelets}, SIAM, Philadelphia, PA,  1992.

\bibitem{Egorychev2011}
G.P. Egorychev, \emph{Combinatorial identity from the theory of
integral representations in $\mathbb C^n$}, Irkutsk Gos. Univ. Mat., 4(4), (2011), pp. 32--44 (in Russian).

\bibitem{Herrmann1971}
O. Herrmann, \emph{On the approximation problem in nonrecursive digital filter design}, IEEE Trans. Circuit Theory Vol. 18 (1971), pp. 411--413.

\bibitem{KoornwinderSchlosser2008}
T.H. Koornwinder, M.J. Schlosser, \emph{On an identity by Chaundy and Bullard. I}, Indag. Math.(N.S.), 19 (2008), pp. 239-–261.

\bibitem{KoornwinderSchlosser2013}
T.H. Koornwinder, M.J. Schlosser, \emph{On an identity by Chaundy and Bullard. II}, Indag. Math.(N.S.), 24 (2013), pp. 174-–180.

\bibitem{KrivokoleskoLeinartas2012}
V.P. Krivokolesko, E.K. Leinartas, \emph{On identities with polynomial coefficients}, Irkutsk Gos. Univ. Mat., 5(3)(2012), pp. 56--63 (in Russian).

\bibitem{LabelleYeh1990}
J. Labelle, Y.N. Yeh, \emph{Generalized Dyck paths}, Discrete
Mathematics. Vol. 82, Issue~1 (1990), pp. 1--6.

\bibitem{LeinartasLyapin2009}
E.K. Leinartas, A.P. Lyapin, \emph{On the Rationality of Multidimentional Recusive Series}, Journal of Siberian Federal University Mathematics \& Physics,  2(4), (2009), pp. 449--455.

\bibitem{Leinartas2007}
E.K. Leinartas, \emph{Multiple Laurent series and fundamental solutions of linear difference equations}, Siberian Math. J. 48(2), (2007), pp. 268--272.

\bibitem{LeinartasRogozina2015}
E.K. Leinartas, M.S. Rogozina, \emph{Solvability of the Cauchy problem for a polynomial difference operator and monomial bases for the quotients of a polynomial ring}, Siberian Math. J. 56(1), (2015), pp. 92--100.

\bibitem{PukhlikovKhovanskii1993}
A.V. Pukhlikov, A.G. Khovanskii, \emph{The
Riemann-Roch theorem for integrals and sums of quasipolynomials on virtual polytopes}, St. Petersburg Mathematical Journal, 4 (1993), no. 4, pp. 789--812.

\bibitem{Stanley1990}
R. Stanley, \emph{Enumerative Combinatorics}, Volume 1, 1990.

\bibitem{Sturmfels1995}
B. Sturmfels, \emph{On vector partition functions}, Journal of Combinatorial Theory. Series A 72 (1995), pp. 302--309.
\end{thebibliography}
\end{document}